\documentclass[12pt,leqno]{amsart}
\usepackage[dvips]{graphics}
\usepackage{amssymb}
\usepackage{verbatim}
\usepackage{hyperref}
\usepackage{color}

\numberwithin{equation}{section}
\newdimen\vintkern\vintkern12pt
\def\vint{-\kern-\vintkern\int}

\newtheorem{thm}{Theorem}[section]

\newtheorem{lem}[thm]{Lemma}
\newtheorem{cor}[thm]{Corollary}
\newtheorem{prop}[thm]{Proposition}

\newtheorem{quest}[thm]{Question}

\newtheorem{rem}[thm]{Remark}

\newcommand{\tref}[1]{Theorem~\ref{#1}}
\newcommand{\cref}[1]{Corollary~\ref{#1}}

\newcommand{\R}{\mathbb{R}}

\begin{document}
	\pagebreak
	
	
	\title{A note on subsets of positive reach}
	
\thanks{
	The author was partially supported by the DFG grants  SFB TRR 191.}
	
	\author{Alexander Lytchak}
	
	\address
	{Institute of Algebra and Geometry\\ KIT\\ Englerstr. 2\\ 76131 Karlsruhe, Germany}
	\email{alexander.lytchak@kit.edu}


	\begin{abstract}
	We  provide  new structural results on sets of positive reach in Euclidean spaces and Riemannian manifolds.
	\end{abstract}
	
	\maketitle

	\renewcommand{\theequation}{\arabic{section}.\arabic{equation}}
	\pagenumbering{arabic}

	\section{Introduction}
	Sets of positive reach in Euclidean spaces were introduced by Herbert Federer \cite{Federer}, as common generalizations of convex subsets and 
$\mathcal C^{1,1}$-submanifolds.  They  have turned out to be relevant   in Riemannian, integral and  metric geometry, cf. 
\cite{Kleinjohann}, \cite{Th}, \cite{KL-subm}, \cite{Hug}.

A subset $C$ of a Riemannian manifold $M$ has \emph{positive reach} if the closest-point projection is uniquely defined in a neighborhood of $C$ in $M$. 	
The notion of positive reach  is invariant under $\mathcal C^{1,1}$-diffeomorphisms and it is local.
Viktor Bangert verified in \cite{Bangert}  that being of positive reach in a Riemannian manifold does not depend on the  choice of the Riemannian metric, but only on the $\mathcal C^{1,1}$-atlas.

  		Geometry and topology of subsets of positive reach  
has been investigated in many  papers including \cite{Federer}, \cite{Hug2},  \cite{Ly-reach},  \cite{Ly-conv}, \cite{Rataj}, \cite{RZ}.   We refer to Section \ref{sec: prel} for a summary of the main properties and recall here only the facts needed to motivate and to state the results of the present paper.

  At any $x\in C$ there is a well-defined tangent cone $T_xC$ which is a convex cone in the Euclidean space $T_x M$.
   The maximal dimension $k$ of the cones $T_xC$ coincides with the Hausdorff dimension of $C$ and is called the \emph{dimension}  of $C$ \cite[Theorem 4.8, Remark 4.15]{Federer}.  
A connected $m$-dimensional  subset $C$ of positive reach is a topological manifold if and only if it is a $\mathcal C^{1,1}$-submanifold, \cite[Proposition 1.4]{Ly-conv}. Moreover, this happens if and only if all
tangent cones $T_xC$ are  Euclidean spaces.

Our results describe what happens if these   equivalent conditions are not satisfied.   The set of positive reach $C$ fails to be a $\mathcal C^{1,1}$-manifold \emph{without boundary} if and only if there is a point with the  infinitesimal structure of a manifold \emph{with boundary}:

\begin{thm} \label{thm: 1}
Let $C$ be a connected  $m$-dimensional subset of positive reach in a Riemannian manifold $M$.
Either $C$ is a $\mathcal C^{1,1}$-submanifold without boundary  of $M$ or, for some 
	$x\in C$,  the tangent cone $T_xC$ is isometric to 
	an $m$-dimensional Euclidean half-space. 
\end{thm}

A slighlty stronger statement can be found in Theorem \ref{thm: 11} below.

 A  neighborhood in $C$ of any point $x\in C$ as in Theorem \ref{thm: 1} is  $\mathcal C^{1,1}$-equivalent to a convex body in $\R^m$, as stated in the following  main result of the present paper:

\begin{thm}	 \label{thm: 2}
	Let $C\subset M$ be an $m$-dimensional subset of positive reach in an $n$-dimensional manifold $M$. Let $x\in C$ be such that the convex cone $T_xC$ has dimension $m$.  

	Then   there exists a closed convex subset $K\subset \R^m$ with non-empty interior in $\R^m$,   an open neighborhood $O$ of $K$ in  $\R^m$    and a $\mathcal C^{1,1}$-embedding  $\Phi:O\to \R^n$  such that $\Phi ( K)$ is a neighborhood of $x$ in $C$.
 \end{thm}

The   statement  is not  quite obvious  even if $\dim M =m$. In this case some forms of Theorem \ref{thm: 2} appear in \cite{Reshetnyak}, \cite{Fu}, \cite[Theorem 5.6]{Hug} and as a statement without proof  in \cite[Appendix B]{Loeper}.

If the tangent cone $T_xC$  in Theorem \ref{thm: 2}  is an $m$-dimensional Euclidean space,
then a neighborhood of $x$ in $C$ is a $\mathcal C^{1,1}$-submanifold  without boundary, \cite[Theorem 7.5]{Rataj}. This implies:

\begin{cor} \label{cor: 1}
Let $C\subset M$ be an $m$-dimensional subset of positive reach.  Then the set $C^+$  of all points $x\in C$ with $m$-dimensional tangent cone $T_xC$ is an open subset of $C$, homeomorphic to an $m$-dimensional manifold with boundary. The boundary of the manifold $C^+$ is the set of points $x\in C^+$  at which $T_xC$ is not a Euclidean space. 
\end{cor}

In other words, $x\in C^+$ is a boundary point  of the manifold $C^+$ if and only if the tangent space $T_xC$ has non-empty boundary, analogous to the structure of boundaries of Alexandrov spaces \cite{P2}.   

\subsection{Proofs and further comments.}
 The proof of Theorem \ref{thm: 1}  is presented in Section \ref{sec: exist} along the following lines.
 We  consider the maximal  non-empty
$m$-dimensional $\mathcal C^{1,1}$ manifold without boundary $U$ contained in $C$. We take an arbitrary point $p\in U$ and find a point  $x$ in the complement $C\setminus U$, which is a closest point 
to $p$ with respect to the intrinsic metric of $C$.  The tangent cone $T_xC$ must then  contain an $m$-dimensional half-space and, therefore, coincide with it.

The proof of Theorem \ref{thm: 2} is more technical. In the full-dimensional case $\dim (C)=\dim (M)$,
the result is essentially contained in the paper \cite{Reshetnyak} preceeding the investigations of sets of positive reach in \cite{Federer}.  We only need to adapt the vocabulary of  \cite{Reshetnyak}
to our situation. 

The case $\dim (M)< \dim (C)$ is proven by finding a $\mathcal C^{1,1}$-diffeomorphism of a neighborhood of $x$ which sends  (a neighborhood of $x$ in) $C$ into an $m$-dimensional submanifold. Then one could apply the previously discussed full-dimensional case.
The construction of the diffeomorphism relies on Whitney's extension theorems and a technical result
obtained in Section \ref{sec: semicont}. This result, Corollary \ref{cor: core}, states that the tangent cones of a set of positive reach vary \emph{Lipschitz semi-continuously} in a precise sense. This semi-continuity may be of some  independent interest.

We finish the introduction with a few comments and questions.

\begin{rem} 
The Riemannian manifold $M$ below and in the formulation of the main results above is always assumed to be smooth. However, all results depends only on the $\mathcal C^{1,1}$-atlas of $M$.
\end{rem}

\begin{rem}
Applications of the  above  results to the theory of submetries will be presented in a separate paper.
\end{rem}

The precise  classification of  (local) structures of  subsets of positive reach up to (biLipschitz) homeomorphisms or even up to $\mathcal C^{1,1}$ diffeomorphisms seems impossible.
However,  it seems possible to obtain reasonable  answers to the following  less ambitious questions.

\begin{quest}
Is it true that any connected 2-dimensional subset  of positive reach is locally 
 biLipschitz equivalent to a subset of positive reach in $\R^2$?
See \cite{Rataj} for an explicit description of subsets of positive reach in the plane.
\end{quest}

\begin{quest}
Can one obtain an infinitesimal chracterization of topological manifolds with boundary among all subsets of positive reach?
\end{quest}

\begin{quest}
Is it possible to describe up to homeomorphisms all germs  of $3$-dimensional subsets of positive reach?
\end{quest}

\noindent{\bf Acknowledgments.} I am grateful to Joe Fu, Daniel Hug and Jan Rataj for helpful comments and discussions.

\section{Preliminaries} \label{sec: prel}
\subsection{Slightly generalized definition and localization}   \label{subsec: localize}  We say that a \emph{locally closed} subset $C$ of a  smooth Riemannian manifold $M$ has  positive reach in $M$ if there exists an open neighborhood $O$ of 
$C$ in $M$ such that the closest-point projection $P^C$ onto $C$ is uniquely defined on $O$.   

This definition is usually given for \emph{globaly closed} subsets $C$.  However,
replacing $O$ by a smaller neighborhood if needed, we may alsways assume that $C$ is closed in $O$.    

The advantage of this generalized definition is the following locality:  For  any subset $C$ of positive reach in $M$, any subset $U$ of $C$, open in $C$, is a subset of positive reach in $M$.  On the other hand, a locally closed subset $C$ of $M$ has positive reach if it is covered by relatively open subsets of positive reach in $M$. 

 For a closed subset $C$ of a manifold $M$, the property of being of positive reach does not depend on the Riemannian metric \cite[Corollary]{Bangert}, moreover, it is invariant under
$\mathcal C^{1,1}$-diffeomorphisms \cite{Bangert}, \cite[Theorem 4.19]{Federer}. Due to the locality stated above, the same statements apply to locally closed subsets.

Given a locally closed subset $C$ of positive reach in $M$ and any point $p\in C$, we can find  a small chart $U$ around $p$, such that $C\cap U$ is closed in $U$. Changing the metric on $U$ to a Euclidean metric, $C\cap U$ becomes a closed subset of positive reach in a Euclidean space.

If  a closed subset $C$ is of positive reach in the Euclidean space $\R^n$ then, for any $p\in C$, the intersection of $C$ with any  sufficiently small closed ball $\bar B_{r} (p)$ is a compact 
contractible subset of positive reach \cite[Theorem 4.10, Remark 4.15]{Federer}, \cite[Lemma 2.3]{Rataj}.

For a \emph{compact} subset of positive reach $C$ in a manifold $M$ there exists a positive number $r$ such that the closest-point projection is uniquely defined on the open $r$-tubular neighborhood  $B_r(C)$ of $C$.  The supremum of such  $r$ is  usually called the \emph{reach of the subset} $C$.

\begin{rem}
Note that for  \emph{non-compact subsets of positive reach} $C$, the number \emph{reach of} $C$ 
defined as above may be $0$. 
\end{rem}

The above consideration allows us to reduce all local statements about arbitrary  subsets of positive reach in a Riemannian manifold to compact connected subsets of positive reach in the Euclidean space.  We will freely use this observation below.

Finally, we refer \cite[4.18]{Federer}, \cite{Bangert}, \cite[Theorem 1.3]{Ly-conv}, \cite[Proposition 1.3]{KL} for many  characterizations of positive reach.

\subsection{Basic properties of subsets of positive reach} \label{subsec: basic}
The topological dimension $\dim (C)$ of a subset of positive reach coincides with its Hausdorff-dimension, \cite[Remark 4.15]{Federer}. Moreover, $\dim (C)$ is the maximum of the dimensions of convex cones $T_xC$.

For $m=\dim (C)$, we denote by $C_{reg}$ the set of all  $x\in C$ such that the tangent cone   $T_xC$ is  isometric to $\R^m$. 
The subset $C_{reg}$ is non-empty,
open in $C$ and it is a $\mathcal C^{1,1}$-submanifold of $M$, \cite[Theorem 7.5]{Rataj}. 

 The complement  $C\setminus \bar {C} _{reg}$ is a locally closed subset of positive reach of dimension at most $m-1$, \cite[Theorem 7.5]{Rataj}.  As in the introduction,  we denote by $C^+ $ the set of points $x\in C$ with $\dim (T_xC)=m$. The previous statement implies $C^+ \subset \bar C_{reg}$.

The tangent cones $T_xC$ depend lower semi-continuously on $x\in C$, \cite[Theorem 4.8]{Federer}, \cite[Proposition 3.1]{Rataj} or Proposition \ref{prop: semi} below: If $x_i$ converge to $x$ in $C$ then
(in any fixed Euclidean chart around $x$) any pointwise Hausdorff limit (of a subsequence)  of convex cones  $T_{x_i}C$ contains the tangent cone $T_xC$.  This immediatly implies:

\begin{lem} \label{lem: open}
Let $C\subset M$ be  of positive reach in $M$ with $\dim (C)=m$. The set  $C^+$
of all $x\in C$ with $\dim (T_xC)=m$   is open in $C$.
\end{lem}

For any point $x\in C$ denote by $\hat T_x C$  the Euclidean subspace of $T_xM$ generated by the convex cone $T_xC$. This is a Euclidean space of the same dimension as $T_xC$.
The semi-continuity of the tangent cones $T_xC$ implies the semi-continuity of the Euclidean spaces $\hat T_xC$.   In case that the dimensions are constant this implies:

\begin{lem}  \label{lem: cont}
Let $C$ be a subset of positive reach in $M$.  Let the sequence $x_i \in C$ converge to $x\in C$. Assume that $\dim (T_{x_i} C)= \dim (T_xC)$ for all sufficiently large  $i$.  Then the linear spaces $\hat T_{x_i} C$ converge to $\hat T_xC$.
\end{lem}

Note that the assumptions on the dimensions are 
satisfied 
if $x\in C^+$.

\subsection{Intrinsic metric on subsets of positive reach}
Let $C$ be a connected of positive reach of a manifold $M$.
The intrinsic metric $d_C$  is defined as usual, \cite[Section 2.3]{BBI01},
by letting $d_C(x,y)$   be the infimum of lengths of curves in $C$ connecting $x$ and $y$.
Any  curve realizing this infimum and parametrized by arclength is called a \emph{$C$-geodesic} between $x$ and $y$.  If $C$ is compact then any pair of points in $C$ is connected by a $C$-geodesic.

Let $C$ be of positive reach in $M$ and let $p\in C$ be arbitrary.
 As observed above, there exists a compact neighborhood $K$ of $p$ in $C$, which is of positive reach in $M$.  We may then change the topology and the metric on $M$ outside a neighborhood $U$ of $K$ in $M$ and embed $U$ isometrically into a \emph{compact} smooth manifold $N$. By this procedure the metric in a neighborhood of $p$ in $M$ and the intrinsic metric in a neighborhood of $p$ in $C$ are not changed.

Applying now \cite[Remark 6.4, Theorem 1.3 ]{Ly-conv}  and  \cite[Theorem 1.1, Theorem 1.2]{Ly-reach} to the pair $K\subset N$ we deduce:

\begin{prop} \label{prop: intr}
Let $C$ be a locally closed subset of positive reach in a manifold $M$. For arbitrary  $p\in C$  and   $\delta >0$ there exist  $\kappa >0$ and $0<r_0< \frac {\pi} 
 {\sqrt \kappa}$ such that for all $r<r_0$ the following hold true:

\begin{itemize}
\item  $\bar B_r(p) \cap C$  is  a compact subset of positive reach in $M$.

\item The intrinsic distance   $d_C$ on $\bar B_r(p) \cap C$ differs from the $M$-distance on 
$\bar B_r(p) \cap C$ at most by the factor $(1+\delta)$.

\item  With respect to  $d_C$, the subset $\bar B_r(p) \cap C$ is convex in $C$.

\item With respect to  $d_C$,  the subset $\bar B_r(p) \cap C$ is a CAT($\kappa$)
space.
\end{itemize}
\end{prop}

The last point above implies that $\bar B_r(p) \cap C$ with respect to the intrisic metric is uniquely geodesic. In particular, it is contractible.

\section{Semi-continuity of tangent spaces} \label{sec: semicont}
We are going to discuss a Lipschitz-version of the semi-continuity of tangent cones  in this section. 
Recall
from
\cite[Theorem 1.2, Theorem 1.3, Example 3.4]{Ly-conv}:  

\begin{lem}  \label{lem: opt}
There exists some universal constant $\mu_1 >0$ with the following property.
If $C$ is a compact subset of $\R^n$ of reach $\geq 1$  then any $C$-geodesic 
$\gamma:[a,b]\to C$  parametrized by arclength is a $\mathcal C^{1,1}$ curve, and 
$\gamma':[a,b]\to \R^n$ is $\mu_1$-Lipschitz continuous. 
\end{lem}

The optimal value of $\mu_1$  does not play a role here and it does not follow from \cite{Ly-conv}, but it might be of some independent interest:

\begin{quest}
What is  the optimal value of $\mu_1$ in Lemma \ref{lem: opt}
and  the optimal value  of  $\mu$ in Propostion \ref{prop: semi} below?
\end{quest}

We can now deduce the following semi-continuity statement:

\begin{prop}  \label{prop: semi}
There exists a universal constant $\mu>0$ with the following property.  Let $C\subset \R^n$ be compact subset  of reach $\geq 1$. Let $\varepsilon \leq 1 $  and $\gamma :[0,\varepsilon] \to C$ be a $C$-geodesic
parametrized by arclength.  Then, for $p=\gamma (0)$, $v= \gamma '(0) \in T_pC$ and any $q\in C$, with $||p-q||\leq \varepsilon ^2$, the  distance from $v$ to  $T_qC$  can be estimated as: 
$$ d(v,T_qC) \leq \mu \cdot ||p-q||\;.$$
\end{prop}

\begin{proof}
Without loss of generality we may assume that $p$ is the origin $0$.
We may further  assume that $||p-q ||=||q|| =\varepsilon^2$, otherwise we just replace $\gamma$ by a shorter subcurve. 

 Set $u= \gamma (\varepsilon)$. Then,  due to Lemma \ref{lem: opt},
  $$||u-\varepsilon \cdot v|| \leq \frac {\mu _1} 2 \varepsilon ^2\;.$$
Hence $$||(u-q) -\varepsilon\cdot v|| \leq (1+ \frac {\mu _1} 2) \varepsilon ^2 \;.$$ 
On the other hand, by \cite[Theorem 4.18]{Federer},
$$d(u-q, T_qC) \leq \frac {||u-q||^2 }  2\;.$$
Since $T_qC$ is a cone, the triangle inequality implies
$$d(v, T_qC)\leq (1+ \frac {\mu _1} 2) \varepsilon +  \frac {||u-q||^2 }  {2\varepsilon}\;.$$
Since $||u||\leq \varepsilon \leq 1$, the second summand is at most $2\varepsilon$.  Thus, we deduce the required inequality with $\mu=3 +  \frac {\mu _1} 2$.
\end{proof}

 We  extend the above conclusion from a single  $v\in T_pC$ to 
large  convex subcones of $T_pC$, more precisely to the set of all vectors 
lying at least at some distance from the boundary of $T_pC$.

 For any
 $\varepsilon >0$,   we consider the set  $T_{p,\varepsilon}C$  of all
$v\in T_pC$, such that the ball of radius $\varepsilon \cdot ||v||$ around $v$ inside the affine hull
$\hat T_pC$ is contained in $T_pC$.  
If $T_pC=\hat T_pC$ then $T_{p,\varepsilon} C= T_pC$, for any $\varepsilon$. In general,
$T_{p,\varepsilon} C$ is a convex subcone of $T_pC$. The subcones $T_{p,\varepsilon} C$ increase with decreasing $\varepsilon$ and their union is the set of inner points of $T_pC$ relative to $\hat T_pC$.

Now we can  deduce from Proposition \ref{prop: semi}:

\begin{cor}  \label{cor: core}
For any compact subset $C$ in $\R^n$ of reach $\geq \delta$ in $\R^n$, for  any point $p\in C$ and  
any $\varepsilon >0$ the following holds true.

There exists some $s=s(p, \varepsilon)>0$ such that for any $q\in C\cap B_s (p)$ and any vector $v$ in the convex subcone $T_{p,\varepsilon}C\subset T_pC$, the distance from $v$ to $T_qC$ is estimated by 
$$ d(v,T_qC) \leq 2\mu \cdot \delta  \cdot ||p-q||\cdot ||v||\;,$$
where $\mu$ is the constant obtained in Proposition \ref{prop: semi}.
\end{cor}

\begin{proof}
Rescaling the space we may assume  $\delta =1$.
 We fix some  $\varepsilon _0 <<\varepsilon$ and adjust it in the course of the proof.

 The unit sphere $S$ in  the  cone  $T_pC$ is the completion  of the set of  starting directions of $C$-geodesics emanating from $p$.
Hence, we find some $t=t(\varepsilon_0)>0$ and $C$-geodesics $\gamma_1,...,\gamma _k :[0,t]\to C$ starting at $p$ in unit directions $v_1,...,v_k$ such that $\{v_1,...,v_k\}$ is
$\varepsilon _0$-dense in $S$.  

  By Proposition \ref{prop: semi}, for any $q\in C$ with $||q-p|| \leq t^2$ we get
\begin{equation} \label{eq: eq}
 d(v, T_qC) \leq \mu \cdot ||p-q||\,,
\end{equation}
for  $v=v_i$, for  any  $i=1,...,k$.  By convexity of the distance function to the convex cone $T_qC$ the 
inequality \eqref{eq: eq} holds true for any $v$ in the convex hull $K_{\epsilon_0}$ of
the unit vectors $v_i$ and the origin $0$.  

For $\varepsilon_0\to 0$ the convex subsets $K_{\varepsilon _0}$ converge to 
the unit ball in $T_pC$.  Thus, for $\varepsilon _0 $ small enough  and any unit vector  $w\in 
T_{p,\varepsilon}C$, the convex subset $K_{\varepsilon _0}$ contains $\frac 1 2 w$. Then
$$d(\frac 1 2 w,T_qC)\leq \mu \cdot ||p-q|| \leq 2\cdot \mu \cdot ||p-q|| \cdot ||\frac 1 2 w||\;.$$

This implies that the statement of the Corollary holds true for $s=t(\varepsilon_0)$, for
such sufficiently small $\varepsilon_0$.
\end{proof}

\section{Existence of boundary points} \label{sec: exist}
The following result is a minor generalization of Theorem \ref{thm: 1}, which is more suitable for localization. Clearly, it implies Theorem \ref{thm: 1}.

\begin{thm} \label{thm: 11}
Let $C$ be an  $m$-dimensional set of positive reach in a Riemannian manifold $M$ of dimension $n$.  Let $C_{reg}$ be the maximal $m$-dimensional $C^{1,1}$-submanifold without boundary contained in $C$. 

Consider the set $\partial C _{reg}$ of boundary points of $C_{reg}$ in $C$.  Consider the subset $G$ of all points $x\in \partial C_{reg}$ with $T_xC$ isometric to an $m$-dimensional
Euclidean half-space. Then $G$ is dense in $\partial C_{reg}$. 
\end{thm}

\begin{proof}
If $\partial C_{reg}$ is  empty, there is nothing to be proven. Thus, assume that $\partial C_{reg}$ is not empty. We fix any $p\in  \partial C_{reg}$ and  $ \varepsilon >0$ and are going to find some $x\in B_{ \varepsilon} (p) \cap G$.

We apply Proposition \ref{prop: intr} with $\delta =1$ and obtain some $r_0 < \varepsilon $ such that, for all $r\leq r_0$, the intersection $\bar B_r(p) \cap  C$ is a compact subset of positive reach in $M$. Moreover,  $\bar B_r(p) \cap  C$   is convex in the intrisic metric $d_C$ of $C$.

Choose $r= \frac {r_0} 4$ and an arbitary $y\in B_{r} (p) \cap C_{reg}$. Consider a closest point $x\in \partial C_{reg}$ to $y$ with respect to the intrinsic distance $d_C$.  Then
$x\in B_{3r} (p) \subset B_{ \varepsilon} (p)$.  It remains to verify $x\in G$. 

Consider the $C$-geodesic  $\gamma:[0,a]\to C$  connecting $x$ and $y$  and parametrized by arclength.  Since $x$ is a closest point  to $y$ on $\partial C_{reg}$,  for
any $0<s\leq a$ the following holds:   The open ball $W_s$ of radius  $s$ around $\gamma (s)$ with respect to  $d_C$ does not intersect $\partial C_{reg}$. Thus, this ball  $W_s$ is completely contained  in the $\mathcal C^{1,1}$-manifold $C_{reg}$.

In particular, any $C$-geodesic  in $W_s$  extends as a 
$C$-geodesic up to points with distance $s$ from $\gamma (s)$, \cite[Theorem 1.5]{LSchr}.  Moreover, $W_s$ 
is uniformly  biLipschitz to the $s$-ball in $\R^m$, since $W_s$ is a geodesically convex $\mathcal C^{1,1}$-manifold which is uniquely geodesic and has curvature uniformly bounded from both sides  \cite[Proposition 1.7]{KL}.

 Identify the tangent cone $T_xC$ at $x$ with the blow-up  of $C$ at $x$  \cite[Remark 6.1]{Ly-conv}.  For the starting direction $v$ of $\gamma$, which is a unit vector in $T_xC$,
we deduce: The open unit ball around $v$ is $m$-dimensional. Moreover,  no geodesic in $T_xC$ terminates at a point with distance less than $1$ to $v$.  Thus,  the convex cone $T_xC$ contains the closed unit $m$-dimensional ball  $W$ around $v$. Since $T_xC$ is a cone,  it contains   the tangent cone $T_0 W$ which is an 
$m$-dimensional Euclidean half-space.

 The tangent cone $T_xC$ is a convex cone of dimension at most $m=\dim (C)$ containing an $m$-dimensional half-space $T_0 W$. Moreover, $T_xC$
is not a Euclidean $m$-dimensional space, since $x$ is not contained in $C_{reg}$. Therefore, $T_xC= T_0W$. Hence  $x\in G$. 
\end{proof}

\section{Structure around the boundary points}
\subsection{Preparation} \label{subsec: prepare}
We are going to prove Theorem \ref{thm: 2} in  this section. 
Thus, let $C$ be an $m$-dimensional  subset of positive reach in an $n$-dimensional Riemannian manifold $M$. Let $x\in C$ be a point
with $\dim (T_xC)=m$.  If $T_x M$ is a Euclidean $m$-dimensional space, then $x$ is contained in the manifold $C_{reg}$, \cite[Theorem 7.5]{Rataj}. The statement of Theorem \ref{thm: 2} follows directly in this case.

We assume from now on that $T_xC$ is  $m$-dimensional, but not a Euclidean space.   The statement of Theorem \ref{thm: 2} is local. Arguing as in Subsection \ref{subsec: localize}, we may assume  $M=\R^n$  and that $C$ is compact.

  We may further assume $x=0$ and that the $m$-dimensional  Euclidean space $V:=\hat T_0C$, generated by the  cone $T_0C$, is the coordinate  subspace
$$V=\R^{m} = \R^m \times \{0\} \subset \R^m \times \R^{n-m}=\R^n \;.$$
We find some unit vector $v_0\in T_0C$ and $\varepsilon >0$ such that  $T_0C$ contains the ball of  radius $4 \varepsilon$ in $V= \R^m$.  After a rotation we may assume that $v_0$ is the last coordinate vector $v_0=e_m$
in $\R^m$.

We fix some $\delta >0$ smaller than   the reach of the compact subset $C$. 
Applying  Lemma \ref{lem: open}, Proposition \ref{prop: intr} and  the semi-continuity of tangent cones, we find some $r>0$ 
with the following properties, for $C'=\bar B_r (0)\cap C$:

\begin{itemize}
\item $C'$ has reach at least $\delta$ in $\R^n$.
\item For any  $p\in  C'$, we have $\dim (T_p C)=m$.
\item  $T_p C$ contains some $m$-dimensional Euclidean ball of radius $2 \varepsilon$ around some unit vector $v_p \in T_pC$.
\item The subset $C'$ is convex with respect to the intrinsic metric $d_C$.
\item The intrinsic metric $d_C$ and the Euclidean metric differ on $C'$ at most by the factor $2$.
\end{itemize}

\subsection{The full-dimensional case}  \label{subsec: full}
Assume first that $m=n$, thus that $C$ is full-dimensional.  In this case the result is essentially contained in \cite{Reshetnyak}, as we are going to explain now.

Consider the family $\mathcal O$ of all closed balls of radius $\delta$ which have exactly one point in common with $C$. This family is closed in the Hausdorff topology and  every point in $\partial C$ lies in some ball $O\in \mathcal O$, compare   \cite[Proposition 3.1 (vi)]{Rataj}.   In terms of \cite{Reshetnyak}, this means that $\partial C$
is an $O_{\delta}^{\ast}$ subset of $\R^m$.

The topological boundary $\partial C$ is precisely the set of points $p\in C$ at which $T_pC$ is not $\R^m$, \cite[Theorem 9.5]{Rataj}.  In particular, $0\in \partial C$.

By assumption $e_m$ is an interior point of $T_0C$. By the semi-continuity of tangent cones,
 $T_pC$ contains $e_m$ as an interior point, for any $p$ in $C$ close to $0$.
Making $r$ smaller, we may assume that $e_m$ is an interior point of $T_pC$,  for all $p\in C'$.

Then, for any $p\in C'  \cap  \partial C$, the intersection of all balls  in the family $\mathcal O$ which contain $p$ is non-empty (more precisely, all these balls contain the point $- t \cdot e_m$, for a sufficiently small $t>0 $). In the notation used in \cite{Reshetnyak}, this means that $C'\cap \partial C$ is an $O_{\delta}$ set and the main results of \cite{Reshetnyak} can be applied.

In particular, we find a ball $W_0$ in   the orthogonal complement $\R^{m-1}$
of $e_m$ with the following property:  For any $w\in W_0$ the line $\gamma _w(t)=w+te_m$
intersects  $C'\cap \partial C$  in at most one point \cite[Theorem 1]{Reshetnyak}.

On the other hand, for all small $t>0$, the point  $\gamma _0(t)$ lies in the interior of $C$ \cite[Lemma 3.5]{Rataj}. Hence, choosing  $W_0$ smaller, if necessary, we may assume 
that $\gamma _w$ intersects the interior of $C'$ for any $w\in W_0$. 

For all small $t<0$, the point   $\gamma _0 (t)$ is not contained in $C$.
Otherwise, $-e_m \in  T_0C$ and, since $e_m$ is an inner point of the convex cone $T_0C$,  this would imply that $T_0C = \R^m$, in contradiction to our assumption $0\notin C_{reg}$. Hence, making $W_0$ smaller if  needed, we may assume that 
$\gamma _w (t_0)$ is not in $C$ for some fixed small $t_0<0$ and any $w\in W_0$.

Therefore, for any $w \in W_0$ the intersection of $\gamma _w$ and $C'$ is a compact
segment $I_w = [\gamma _w(f_w), \gamma _w (g_w)]$, for some $f_w  <g_w$.  Moreover, $\gamma _w(f_w)\in \partial C$ and $\gamma _w(g_w)$ lies on the sphere 
$\partial B_r(0)$ and $g_w>0$.

By compactness of $\partial C$, we deduce that the map $w\to f_w$ is continuous.
Thus, near the origin $0$, the set $C$ is given as the supergraph in $\R^m=\R^{m-1} \times \R$ of the continuous function 
$f:W_0\to \R$, $f(w):=f_w$.   

Applying \cite[Theorem 5]{Reshetnyak}, we deduce that the function $f$ is semi-convex  on $W_0$, thus there exists some $c\in \R$ such that the function
$$\hat f(w):= f(w) + c\cdot ||w||^2$$ 
is convex on $W_0$.

Consider the diffeomorphism $\Phi :\R^m =\R^{m-1}\times \R \to \R^{m-1}\times \R= \R^m$  
$$\Phi (w,t):=( w, t + c  \cdot ||w^2||)\;.$$

 Thus, $K=\Phi (C) \cap  (W_0\times \R)$ is given in a neighborhood of $0$ as the supergraph of the convex function
$\hat f$. Thus, the smooth diffeomorphism $\Phi^{-1}$  sends a neighborhood of $0$
in the convex set $K$  onto a neighborhood of $0$ in $C$. 
This finishes the proof of Theorem \ref{thm: 2} in case $m=n$.

\subsection{Lipschitz continuity of tangent spaces} 
We turn to the general case $n\geq m$  and  fix   $ C'\subset C, \varepsilon, \delta , r$  as in Subsection \ref{subsec: prepare}. 
As in Subsection \ref{subsec: basic}, we denote by $\hat T_p C$ the $m$-dimensional Euclidean subspace of $T_pM$ generated by the cone $T_pC$, for any $p\in  C'$.
  We are going to deduce    from   Corollary \ref{cor: core} the following strengthening of Lemma \ref{lem: cont}:

\begin{prop} \label{prop: Lip}
The map $p\to \hat T_p C$ from $C'$ to the Grassmanian $\mathrm {\bf Gr}_{m,n}$ of $m$-dimensional subspaces of $\R^n$ is  Lipschitz continuous.
\end{prop}

Before we embark on the proof,  recall that one (of many biLipschitz equivalent) metric on the Grassmannian $\mathrm {\bf Gr}_{m,n}$ is defined by  
$$d(U,W) \leq \sup _{u\in U, ||u||=1} d(u,W)\;,$$
for $U,W\in \mathrm {\bf Gr}_{m,n}$.  The distance $d(u,W)$ is the norm $||P(u)||$, where $P=P^{W^{\perp}}$ denotes the projection onto the orthogonal complement $W^{\perp}$ of $W$.  
The linearity of $P$ directly implies the following \newline
\emph{Claim}: If $u_0\in U$ is a unit vector and for any  $u$ in the ball   $B_{\varepsilon} (u_0)\cap U$  we have $d(u,W)\leq t$ then 
$$d(U,W) \leq \frac {2t} { \varepsilon}\,.$$

Now we can proceed with

\begin{proof}[Proof of Proposition \ref{prop: Lip}]
On $C'$ the intrinsic metric is $2$-biLipschitz to the induced one. Thus, it suffices to prove
the Lipschitz property pointwise, hence, to verify the following
\newline
\emph{Claim:} 
 There is some $\lambda =\lambda (\varepsilon, \delta)>0$ such that,
for any $p\in C'$, 
$$\limsup_{q\in C, q\to p} \frac {d(\hat T_p C, \hat  T_q C) } {||p-q||} \leq \lambda\;.$$ 

By Corollary \ref{cor: core}, for any $v$ in the ball $O$  in $T_pC$ of radius $\varepsilon$ around   a unit   vector $v_p$ and all  $q$, sucfficiently close to $p$, we have 
$$d(v, T_qC)\leq 4\mu \cdot \delta \cdot ||p-q||\;,$$
with some universal constant $\mu$.
The observation preceding the proof of the Proposition  implies
$$d(\hat T_pC, \hat T_qC)\leq  \frac {8\mu } {\varepsilon} \cdot  \delta \cdot ||p-q||\,,$$
for all such $q$.  
This shows the claim with $\lambda=  \rho \cdot 2\mu \cdot \delta$
and finishes the proof of Proposition \ref{prop: Lip}.
\end{proof}

\subsection{Finding a larger submanifold}
In the setting of Subsection \ref{subsec: prepare} we are going to show

\begin{lem} \label{lem: last}
There exists   $s>0$  and a  $\mathcal C^{1,1}$-diffeomorphism $\Psi$  between two  neighborhoods of $x=0$  in $\R^n$, such that 
$\Psi (C' \cap \bar B_s(0) ) \subset V=\R^m$.
\end{lem}

Once the Lemma is verifed, the image
$Q:=\Psi (C'  \cap \bar B_s(0)  )$ is  an $m$-dimensional  subset of positive reach in $V$ 
\cite[Theorem 4.19]{Federer}.
 Applying   
the result in the  case $m=n$  obtained in Subsection \ref{subsec: full},
we find a diffeomorhism $\Phi _0 $ between  two neighborhoods $O_1,O_2$ of $0$ in $V$, which sends $O_1\cap Q$ onto a convex subset $K$.  Then the statement of 
Theorem \ref{thm: 2} follows by taking $\Phi$ to be the composition $\Phi _0$ and $\Psi^{-1}$.

Therefore it remains to  provide:
\begin{proof}[Proof of Lemma \ref{lem: last}]
Rescaling the space we may assume that the reach $\rho$ is at least $1$. We fix some $\lambda \geq 1$ such that the map  $p\to \hat T_p C$ is $\lambda$-Lipschitz on $C'$, 
Proposition \ref{prop: Lip}.

For any  $0<s<\frac 1 {4\lambda } \leq \frac 1 4 $ and  for any $p\in C' \cap \bar B_s (0)$ we have
$$d(\hat T_p C, V)  \leq \frac 1 {4}\;.$$

We may replace $r$ by $s$ in the definition of $C'$ and assume that $C'= C' \cap \bar B_s (0)$. For any $p \neq q\in C'$ we deduce from  the above inequality and \cite[Theorem 4.8]{Federer} that 
$$d(\frac {p-q} {||p-q||}, V)  \leq d(\frac {p-q} {||p-q||}, \hat T_q C)  + d(\hat T_qC, V)\leq
\frac {s} 2 + \frac {1} {4} < \frac {1} {2}\;.$$

Consider the orthogonal projection $P:C'\to V$ and    denote by
$Q\subset V$  the compact image $Q:=P(C')$.
The above inequality implies that the $1$-Lipschitz map  $P:C'\to Q$ is injective and has a 2-Lipschitz continuous inverse $\Phi =P^{-1}:Q\to C'$.


  We claim that it is  sufficient to find an extension of $\Phi$ to a $\mathcal C^{1,1}$-map $\Phi:V\to \R^n$.   Once this is done, the differential of $\Phi$ at $0$ will automatically be the identity. Hence, the map $\Phi$ will be a 
$\mathcal C^{1,1}$-embedding  of a neighborhood $O_1$ of $0$ in $V$ into $\R^n$.
 Then, making $O_1$ smaller if necessary, we can extend $\Phi$ to a diffeomorphism between two neighborhoods of $0$ in $\R^n$.   In this case, the inverse map $\Psi$
will be a $\mathcal C^{1,1}$ diffeomorphism between two neighborhoods of $0$ in $\R^n$
and such that $\Psi |_{C'} =P$.  This would finish the proof.

It remains to find the extension of $\Phi :Q\to C'$ to a $\mathcal C^{1,1}$-map $\Phi:V=\R^m\to \R^n$.  In order to do so, we will rely on Whitney's extension theorem \cite{Whitney} in the form of \cite{Glaeser}.

For any $z\in Q$ with $\bar z=\Phi (z)$, consider the inverse  map $f^1_z: V\to \hat T_{\bar z} C \subset \R ^n$  of the linear isomorphism $P:\hat T_{\bar z} C  \to V$.  Denote by $f_z:V\to \R^n$ the affine map ("Taylor plynomial of degree one")
$$f_z(q):=\Phi (z) + f^1_z (q-z)\,.$$
Whitney's extension theorem in the form of \cite[Proposition VII]{Glaeser} implies the following. There exists an extension of $\Phi:Q\to \R^n$ to a $\mathcal C^{1,1}$-map $\Phi:\R^m \to \R^n$ with $D_z \Phi = f^1_z$ for all $z\in Q$ if and only if for some $c>0$  and all $y,z \in $ the  two subsequent conditions are satisfied:
\begin{itemize}

\item  $||f^1_z -f^1_y || \leq c \cdot ||z-y|| \;.$

\item  $||f_z (z)- f_y(z)||  =|| \Phi (z)- \Phi (y) - f^1_y (z-y)||\leq c \cdot ||z-y||^2$.

\end{itemize}

Set $\bar z:= \Phi (z)$ and $\bar y:= \Phi (y)$.

In order to verify the first item above, consider an arbitrary unit vector $v\in V$. Then the vector $f_z^1 (v) \in \hat T_{\bar z} C$ has norm at most $2$. By Proposition \ref{prop: Lip} we can find some vector
$w\in \hat T_{\bar y} C$ with 
$$||w-f^1_z (v)|| \leq 2\cdot \lambda \cdot ||\bar z- \bar y ||   \leq 4\cdot \lambda \cdot ||z-y ||\,,$$
where $\lambda$ is the Lipschitz constant  provided in  Proposition \ref{prop: Lip}.
Then 
$$||P(w)-v||=||P(w-f^1_z(v))|| \leq 4\cdot \lambda \cdot ||z-y ||\,.$$
Hence 
$$||w-f^1_y (v)||= ||f^1_y (P(w) -v)|| \leq  8\cdot \lambda \cdot ||z-y ||$$
With $c:=12 \lambda$ the triangle inequaliy implies
  $$||f^1_z(v) -f^1_y (v) || \leq c \cdot ||z-y|| \;.$$
Since $v$ was an arbitrary unit vector, this shows the first item.

In order to verify the second item, we apply \cite[Theorem 4.18]{Federer}  and find a vector $w\in T_{\bar y}  C\subset \hat T_y C$ with
$$w- (\bar z -\bar y) \leq \frac 1 2 ||\bar z  -\bar y|| ^2 \leq 2 \cdot ||z-y||^2$$
Therefore
$$||P(w)- (z-y)||\leq 2 \cdot ||z-y||^2\;.$$
Hence
$$||w-f^1_y (z-y)||\leq 4\cdot ||z-y||^2\;.$$
By the triangle inequality we deduce
$$||\bar z- \bar y -f_y^1 (z-y)|| \leq 6\cdot ||z-y||^2\;.$$
Thus, for $c=\max \{12\lambda, 6\} $ this finishes the verification of both conditions,  the proof of Lemma \ref{lem: last} and of Theorem \ref{thm: 2}. 
\end{proof}

\bibliographystyle{alpha}
\bibliography{reach}

\begin{thebibliography}{Whi34}

\bibitem[Ban82]{Bangert}
V.~Bangert.
\newblock Sets with positive reach.
\newblock {\em Arch. Math. (Basel)}, 38(1):54--57, 1982.

\bibitem[BBI01]{BBI01}
D.~Burago, Y.~Burago, and S.~Ivanov.
\newblock {\em A course in metric geometry}, volume~33 of {\em Graduate Studies
  in Mathematics}.
\newblock American Mathematical Society, Providence, RI, 2001.

\bibitem[Fed59]{Federer}
H.~Federer.
\newblock Curvature measures.
\newblock {\em Trans. Amer. Math. Soc.}, 93:418--491, 1959.

\bibitem[Fu85]{Fu}
J.~Fu.
\newblock Tubular neighborhoods in {E}uclidean spaces.
\newblock {\em Duke Math. J.}, 52(4):1025--1046, 1985.

\bibitem[Gla58]{Glaeser}
G.~Glaeser.
\newblock \'{E}tude de quelques alg\`ebres tayloriennes.
\newblock {\em J. Analyse Math.}, 6:1--124; erratum, insert to 6 (1958), no. 2,
  1958.

\bibitem[HS22]{Hug}
D.~Hug and M.~Santilli.
\newblock Curvature measures and soap bubbles beyond convexity.
\newblock {\em Adv. Math.}, 411(part A):Paper No. 108802, 2022.

\bibitem[Hug98]{Hug2}
D.~Hug.
\newblock Generalized curvature measures and singularities of sets with
  positive reach.
\newblock {\em Forum Math.}, 10(6):699--728, 1998.

\bibitem[KL21]{KL}
V.~Kapovitch and A.~Lytchak.
\newblock Remarks on manifolds with two-sided curvature bounds.
\newblock {\em Anal. Geom. Metr. Spaces}, 9(1):53--64, 2021.

\bibitem[KL22]{KL-subm}
V.~Kapovitch and A.~Lytchak.
\newblock The structure of submetries.
\newblock {\em Geom. Topol.}, 26(6):2649--2711, 2022.

\bibitem[Kle81]{Kleinjohann}
N.~Kleinjohann.
\newblock N\"{a}chste {P}unkte in der {R}iemannschen {G}eometrie.
\newblock {\em Math. Z.}, 176(3):327--344, 1981.

\bibitem[LS07]{LSchr}
A.~Lytchak and V.~Schroeder.
\newblock Affine functions on {${\rm CAT}(\kappa)$}-spaces.
\newblock {\em Math. Z.}, 255(2):231--244, 2007.

\bibitem[LV10]{Loeper}
G.~Loeper and C.~Villani.
\newblock Regularity of optimal transport in curved geometry: the nonfocal
  case.
\newblock {\em Duke Math. J.}, 151(3):431--485, 2010.

\bibitem[Lyt04]{Ly-reach}
A.~Lytchak.
\newblock On the geometry of subsets of positive reach.
\newblock {\em Manuscripta Math.}, 115(2):199--205, 2004.

\bibitem[Lyt05]{Ly-conv}
A.~Lytchak.
\newblock Almost convex subsets.
\newblock {\em Geom. Dedicata}, 115:201--218, 2005.

\bibitem[Per91]{P2}
G.~Perelman.
\newblock {A. D. Alexandrov} spaces with curvature bounded below {II}.
\newblock {\em preprint}, 1991.

\bibitem[Res56]{Reshetnyak}
Yu.~G. Reshetnyak.
\newblock On a generalization of convex surfaces.
\newblock {\em Mat. Sb. N.S.}, 40(82):381--398, 1956.

\bibitem[RZ17]{Rataj}
J.~Rataj and L.~Zaj\'{\i}\v{c}ek.
\newblock On the structure of sets with positive reach.
\newblock {\em Math. Nachr.}, 290(11-12):1806--1829, 2017.

\bibitem[RZ19]{RZ}
J.~Rataj and M.~Z\"{a}hle.
\newblock {\em Curvature measures of singular sets}.
\newblock Springer Monographs in Mathematics. Springer, Cham, 2019.

\bibitem[Tha08]{Th}
Ch. Thaele.
\newblock 50 years sets with positive reach---a survey.
\newblock {\em Surv. Math. Appl.}, 3:123--165, 2008.

\bibitem[Whi34]{Whitney}
H.~Whitney.
\newblock Analytic extensions of differentiable functions defined in closed
  sets.
\newblock {\em Trans. Amer. Math. Soc.}, 36(1):63--89, 1934.

\end{thebibliography}

\end{document}